\documentclass[11pt]{amsart}

\usepackage{amssymb,amsmath,amsthm,newlfont,enumerate}
\usepackage{booktabs}
\usepackage{subfig}
\usepackage[breaklinks=true]{hyperref}
\usepackage{breakcites}
\usepackage{float}
\usepackage{graphicx}
\usepackage{subfig}
\usepackage{epstopdf}
\usepackage{color}

\theoremstyle{plain}
\newtheorem{theorem}{Theorem}[section]

\newtheorem{lemma}[theorem]{Lemma}
\newtheorem{proposition}[theorem]{Proposition}

\theoremstyle{definition}

\newtheorem{example}[theorem]{Example}

\theoremstyle{remark}
\newtheorem*{remark}{Remark}

\numberwithin{figure}{section}
\numberwithin{table}{section}

\newcommand{\RR}{{\mathbb R}}
\newcommand{\bu}{{\mathbf u}}
\newcommand{\bX}{{\mathbf X}}
\newcommand{\cD}{{\mathcal D}}
\newcommand{\cE}{{\mathcal E}}
\newcommand{\cG}{{\mathcal G}}
\newcommand{\cH}{{\mathcal H}}
\newcommand{\cI}{{\mathcal I}}

\renewcommand{\hat}{\widehat}
\renewcommand{\tilde}{\widetilde}
\DeclareMathOperator{\diag}{diag}

\begin{document}

\title[Testing for invariance under group actions]{Application of the Cram\'er--Wold theorem to testing for invariance under group actions}

\author[Fraiman]{Ricardo Fraiman}
\address{Centro de Matem\'atica, Facultad de Ciencias, Universidad de la Rep\'ublica, Uruguay}
\email{rfraiman@cmat.edu.uy}

\author[Moreno]{Leonardo Moreno}
\address{Instituto de Estad\'{i}stica, Departamento de M\'etodos Cuantitativos, FCEA, Universidad de la Rep\'ublica, Uruguay.}
\email{mrleo@iesta.edu.uy}

\author[Ransford]{Thomas Ransford}
\address{D\'epartement de math\'ematiques et de statistique, Universit\'e Laval,
Qu\'ebec City (Qu\'ebec),  Canada G1V 0A6.}
\email{ransford@mat.ulaval.ca}

\begin{abstract}

We address the problem of testing for the invariance of a probability measure under the action of a group of linear transformations.
We propose a procedure based on consideration of one-dimensional projections, justified using a variant of the Cram\'er--Wold theorem. 
Our test procedure is powerful, computationally efficient,
and  dimension-independent, 
extending even to the case of infinite-dimensional spaces (multivariate functional data).
It includes, as special cases, tests for exchangeability and sign-invariant exchangeability. 
We compare our procedure with some previous proposals in these cases, in a small simulation study. 
The paper concludes with two real-data examples.
\end{abstract}



\date{18 May 2022}

\maketitle


\section{Introduction}

Let $d\ge2$ and let $P$ be a Borel probability measure on $\RR^d$.
The \emph{image} or \emph{push-forward} of $P$ under a linear map $T:\RR^d\to\RR^d$
is the Borel probability measure $PT^{-1}$ on $\RR^d$ defined by
\[
PT^{-1}(B):=P(T^{-1}(B)).
\]
The measure $P$ is said to be \emph{$T$-invariant} if $PT^{-1}=P$.
If $G$ is a group of invertible linear self-maps of $\RR^d$, then
$P$ is \emph{$G$-invariant} if it is $T$-invariant for each $T\in G$.

For example, if $G$ is the group of $d\times d$
permutation matrices, then $P$ is $G$-invariant if and only if
it is an exchangeable distribution. If $G$ is the group of signed permutation
matrices, then $P$ is $G$-invariant if and only if it is a sign-invariant exchangeable
distribution. (These terms will be defined in detail later in the article.) 
Of course one can imagine many more examples.

The purpose of this paper is to develop a methodology for testing whether a given
probability measure $P$ on $\RR^d$ is $G$-invariant for a given group~$G$.
We know of no previous work on this subject in this degree of generality,
though certain special cases, notably testing for exchangeability or sign-invariant exchangeability,
have been extensively treated in the literature. We shall provide detailed references when we come to discuss these special cases later in the paper.

A potential obstacle is the fact that $G$ may be quite large, possibly even infinite.
Another possible difficulty is that the dimension $d$ of the underlying space may be very large.
To circumvent these problems, we exploit two devices.

The first idea is a very simple one, namely that, to test for $G$-invariance, it suffices to
test for $T$-invariance as $T$ runs through a set of generators for~$G$. The point is that
$G$ may be generated by a very small set of $T$, even though $G$ itself is large or even infinite.
This idea is explored in \S\ref{S:generators}.

The second idea is that, to test whether $P=PT^{-1}$, 
one can project $P$ and $PT^{-1}$
onto a randomly chosen one-dimensional subspace,
and test whether the projected measures are equal.
This reduces a $d$-dimensional problem to a one-dimensional one,
for which well-known techniques are available.
The justification for this procedure is a variant of the 
Cram\'er--Wold theorem. This idea is described in detail in \S\ref{S:CW}.

Based on these two ideas, we develop a testing procedure in \S\ref{S:testing}.
Under suitable hypotheses on $P$ and $G$, our test is consistent and distribution-free.
We also consider a variant of this procedure, based on projecting onto several randomly-chosen one-dimensional
subspaces. This has the effect of increasing the power of the test.

There follows a discussion of three special cases.
The case of exchangeable distributions is treated in \S\ref{S:exch}
and that of sign-invariant exchangeable distributions in \S\ref{S:signexch}.
We describe the background, perform some simulations, and,
in the case of exchangeability, we compare our results with those
obtained by other techniques in the literature. 

The third case, treated in \S\ref{S:infdim},
illustrates the flexibility of our method. In this case,
$\RR^d$ is replaced by an infinite-dimensional Hilbert space.
The necessary adjustments to our method are described,
and illustrated with a further simulation. 

The article concludes  with  two  examples drawn from real datasets,
one  concerning biometric measurements, and the other satellite images.


\section{Generators}\label{S:generators}

Recall that a group $G$ is said to be \emph{generated} by a subset $A$ if 
every element $g$ of $G$ can be written as a finite product $g=g_1\cdots g_m$,
such that, for each $j$, either $g_j\in A$ or $g_j^{-1}\in A$.
Equivalently, $G$ is generated by $A$ if $A$ is contained in no proper subgroup of $G$.

In our context, the interest of this notion stems from the following simple proposition.

\begin{proposition}
Let $P$ be a Borel probability measure on $\RR^d$,
let $G$ be a group of invertible linear maps $T:\RR^d\to\RR^d$,
and let $A$ be a set of generators for $G$.
Then $P$ is $G$-invariant if and only if $P$ is $T$-invariant for each $T\in A$.
\end{proposition}

\begin{proof}
Define
\[
G_0:=\{T\in G: PT^{-1}=P\}.
\]
Obviously $G_0$ contains the identity. Also, it is easily checked that,
if $T_1,T_2\in G_0$, then also $T_1T_2\in G_0$ and $T_1^{-1}\in G_0$.
Therefore $G_0$ is a subgroup of $G$. By assumption, $G_0$ contains $A$.
As $A$ generates $G$, it follows that $G_0=G$.
\end{proof}

\begin{example}\label{Ex:2gen}
Let $G$ be the group of $d\times d$ permutation matrices
(i.e.\ matrices with one entry $1$ in each row and each column, and zeros elsewhere).
Thus each $T\in G$ permutes the basis vectors of $\RR^d$,
say $Te_j=e_{\sigma(j)}$,
where $\sigma$ is permutation of $\{1,2,\dots,d\}$.
The correspondence $T\leftrightarrow\sigma$ is an isomorphism
between $G$ and $\Sigma_d$, the group of permutations of $\{1,2,\dots,d\}$.
It is well known that, even though $\Sigma_d$ 
contains $d!$ permutations, it can be generated using just two permutations,
for example the transposition $\sigma_1=(1,2)$ and the cycle $\sigma_2=(1,2,3,4,...,d)$
(see e.g.\ \cite[Example~2.30]{Ro78}). 
Thus $G$ has a generating set consisting of just two matrices.
\end{example}

\begin{example}\label{Ex:sign}
Let $G$ be the group of $d\times d$ signed permutation matrices
(i.e.\ matrices with one entry $\pm 1$ in each row and each column, and zeros elsewhere).
This is sometimes called the hyperoctahedral group,  denoted~$B_d$. 
It is the group of symmetries of the $d$-dimensional cube. It contains $d!2^d$ elements,
but, like the symmetric group $\Sigma_d$, it can be generated by just two elements.
However, these elements are more complicated to describe,
see \cite[Proposition~6]{Ja04}.
On the other hand, it is easy to give a set of three generators:
for example, one can take $\{T_1,T_2,T_3\}$, 
where $T_1,T_2$ are the matrices corresponding to the same two permutations $\sigma_1,\sigma_2$ as before,
and $T_3$ is the diagonal matrix given by $T_3:=\diag(-1,1,1,\dots,1)$.
\end{example}


\section{Cram\'er--Wold theorems}\label{S:CW}

In this section we recall the Cram\'er--Wold theorem, as well as a more recent extension.
We also discuss the sharpness of this result in the context testing for $G$-invariance.

Let $P,Q$ be Borel probability measures on $\RR^d$, where $d\ge2$.
We denote by $\cE(P,Q)$ the set of vectors $x\in\RR^d$
such that $P\pi_x^{-1}=Q\pi_x^{-1}$, where $\pi_x:\RR^d\to\RR^d$ is the orthogonal projection onto the one-dimensional subspace $\RR x$ spanned by $x$.
Equivalently, $\cE(P,Q)$ is the set of $x\in\RR^d$ such that $\phi_P(tx)=\phi_Q(tx)$ for all $t\in\RR$,
where $\phi_P,\phi_Q$ denote the characteristic functions of $P,Q$ respectively.
The set $\cE(P,Q)$ is a closed cone (not necessarily convex) in $\RR^d$. 
For detailed proofs of all these facts, see \cite[\S2]{CFR07}.

The following result is a restatement in our notation of a 
well-known theorem of Cram\'er and Wold (\cite{CW36}).

\begin{theorem}[\protect{\cite[Theorem~I]{CW36}}]\label{T:CW}
Let $P,Q$ be Borel probability measures on~$\RR^d$, where $d\ge2$.
If $\cE(P,Q)=\RR^d$, then $P=Q$
\end{theorem}

There are several extensions of this theorem,
in which one assumes more about the nature of the measures $P,Q$ 
and less about the size of $\cE(P,Q)$.
Articles on this subject include those of R\'enyi \cite{Re52},
Gilbert \cite{Gi55}, Heppes \cite{He56}, B\'elisle--Mass\'e--Ransford \cite{BMR97}
and Cuesta-Albertos--Fraiman--Ransford \cite{CFR07}.
We cite one such result, taken from \cite{CFR07}.

\begin{theorem}[\protect{\cite[Corollary~3.2]{CFR07}}]\label{T:CWgen}
Let $P,Q$ be Borel probability measures on~$\RR^d$, where $d\ge2$. Assume that
the absolute moments $m_N:=\int\|x\|^N\,dP(x)$ are all finite and satisfy 
\begin{equation}\label{E:Carleman}
\sum_{N\ge1}m_N^{-1/N}=\infty.
\end{equation}
If the set $\cE(P,Q)$ is of positive Lebesgue measure in $\RR^d$,
then $P=Q$.
\end{theorem}

The moment condition \eqref{E:Carleman} is slightly less restrictive than demanding that the moment generating function of $P$ be finite.
Just how essential is this condition?
The brief answer is that, without it, Theorem~\ref{T:CWgen} breaks down dramatically.
Indeed, given a moment sequence $(m_N)$ that fails to satisfy \eqref{E:Carleman},
one can find probability measures $P,Q$ whose moments are bounded by $m_N$,
and $\cE(P,Q)$ has positive measure (indeed it even has non-empty interior),
yet $P\ne Q$.
See \cite{CFR07} for  an extensive discussion of this topic.

However, we are eventually going to apply Theorem~\ref{T:CWgen} in the special case when 
$Q=PT^{-1}$, where $T:\RR^d\to\RR^d$ is an invertible linear map.
It is \emph{a priori} conceivable that Theorem~\ref{T:CWgen}  might be improvable
for pairs of measures $(P,Q)$ of this particular form.
The following sharpness result, which we believe to be new, shows that this is not so,
at least in the case when $T$ is of finite order.

\begin{theorem}\label{T:sharpexch}
Let $T:\RR^d\to\RR^d$ be a linear map
such that $T\ne I$ but $T^n=I$ for some $n\ge2$.
Let $(M_N)_{N\ge0}$ be a positive sequence satisfying
\[
M_0=1,
\quad
M_N^2\le M_{N-1}M_{N+1}~(N\ge1)
\quad\text{and}\quad
\sum_{N\ge1}M_N^{-1/N}<\infty.
\]
Then there exists  a Borel probability measure $P$ on $\RR^d$ such that
\begin{itemize}
\item  $\int\|x\|^N\,dP(x)\le M_N$ for all $N\ge0$,
\item the cone $\cE(P,PT^{-1})$ is of positive measure in $\RR^d$, but
\item $PT^{-1}\ne P$.
\end{itemize}
\end{theorem}

\begin{remark}
The condition that $T^n=I$ for some $n\ge 2$ is automatically satisfied
if $T$ belongs to a finite group $G$, as will be the case in all the examples that we 
shall study. The article \cite{Ko03} contains a criterion for when $T^n=I$.
\end{remark}

The main new idea in Theorem~\ref{T:sharpexch} is the construction given in 
the following lemma.

\begin{lemma}\label{L:sharpexch}
Let $Q,R$ be Borel probability measures on $\RR^d$.
Let $T:\RR^d\to\RR^d$ be a linear map such that $T^n=I$ for some $n\ge2$.
Define
\[
P:=\frac{1}{n}\Bigl(Q+\sum_{j=1}^{n-1}RT^{-j}\Bigr).
\]
Then $P$ is a Borel probability measure on $\RR^d$ and, writing $\cE_0:=\cE(Q,R)$, we have
\begin{equation}\label{E:sharpexch}
\cE_0\cap T^{-1}(\cE_0)\subset \cE(P,PT^{-1})\subset \RR^d\setminus (\cE_0\bigtriangleup  T^{-1}(\cE_0)).
\end{equation}
\end{lemma}

\begin{proof}
Clearly $P$ is a Borel probability measure on $\RR^d$.
Also, since $T^n=I$, the measure $\tilde{P}:=\sum_{j=0}^{n-1}RT^{-j}$ is $T$-invariant, 
so, since $P=(1/n)(Q-R+\tilde{P})$, we get
\[
P-PT^{-1}=\frac{1}{n}\Bigl(Q-R-(Q-R)T^{-1})\Bigr).
\]
Using the characterization of $\cE$ in terms of characteristic functions, it follows that
\begin{equation}\label{E:phiequiv}
\begin{aligned}
x\in\cE(P,PT^{-1})          
&\iff \phi_{P-PT^{-1}}(tx)=0 \quad\forall t\in\RR\\
&\iff \phi_{(Q-R)}(tx)-\phi_{(Q-R)T^{-1}}(tx)=0 \quad\forall t\in\RR\\
&\iff \phi_{(Q-R)}(tx)-\phi_{(Q-R)}(tTx)=0 \quad\forall t\in\RR.
\end{aligned}
\end{equation}

To prove the first inclusion in \eqref{E:sharpexch}, 
let $x\in\cE_0\cap T^{-1}(\cE_0)$. 
Then both $x,Tx\in\cE_0=\cE(Q,R)$, 
so both $\phi_{(Q-R)}(tx)=0$ and $\phi_{(Q-R)}(tTx)=0$ for all $t\in\RR$. 
By the equivalence \eqref{E:phiequiv}, it follows that $x\in\cE(P,PT^{-1})$. 
This establishes the first inclusion.

For the second inclusion in \eqref{E:sharpexch},
let $x\in\cE_0\bigtriangleup  T^{-1}(\cE_0)$.
Then exactly one of $x$ and $Tx$ lies in $\cE_0=\cE(Q,R)$,
so there exists $t\in\RR$ such that exactly one of $\phi_{(Q-R)}(tx)$
and $\phi_{(Q-R)}(tTx)$ is zero. 
Their difference is therefore non-zero, 
so, by the equivalence \eqref{E:phiequiv} again, 
it follows that $x\notin \cE(P,PT^{-1})$. 
This establishes the second inclusion and completes the proof of the lemma.
\end{proof}

\begin{proof}[Proof of Theorem~\ref{T:sharpexch}]
Let $B$ be a closed ball in $\RR^d$ such that $0\notin B$.
By \cite[Theorem~5.4]{BMR97}, 
there exist mutually singular Borel probability measures $Q,R$ on $\RR^d$
such that both $\int\|x\|^N\,dQ(x)\le M_N$ and $\int\|x\|^N\,dR(x)\le M_N$ for all $N\ge0$,
and also such that $\cE_0:=\cE(Q,R)$ contains all lines not meeting $B$.
Since $T\ne I$,
by choosing $B$ appropriately we may ensure that $\cE_0\ne T^{-1}(\cE_0)$
and also that $\cE_0\cap T^{-1}(\cE_0)$ has positive Lebesgue measure.

Let $P$ be the probability measure constructed from $Q,R$ in Lemma~\ref{L:sharpexch}.
By the left-hand inclusion in \eqref{E:sharpexch}, $\cE(P,PT^{-1})$ has positive Lebesgue measure.
Also, by the right-hand inclusion, $\cE(P,PT^{-1})$ is a proper subset of~$\RR^d$, so $P\ne PT^{-1}$.

It remains to check that $P$ satisfies the moment inequality. 
From the definition of $P$ in Lemma~\ref{L:sharpexch}, we have
\begin{align*}
\int \|x\|^N\,dP(x)
&=\frac{1}{n}\Bigl(\int \|x\|^N\,dQ(x)+\sum_{j=1}^{n-1}\int\|T^j(x)\|^N\,dR(x)\Bigr)\\
&\le \frac{1}{n}\Bigl(M_N+\sum_{j=1}^{n-1}\|T\|^{jN}M_N\Bigr)\le \|T\|^{nN}M_N.
\end{align*}
This is almost what we want. 
Repeating the argument with $(M_N)_{N\ge0}$ replaced by  $(\|T\|^{-nN}M_N)_{N\ge0}$
at the outset, we get $\int\|x\|^N\,dP(x)\le M_N$.
\end{proof}


\section{General testing procedures}\label{S:testing}

\subsection{A  distribution-free procedure}\label{SS:testing1}

Consider the following set-up.
Let $d\ge2$, let $P$ be a Borel probability measure on $\RR^d$,
and let  $G$ be a finite multiplicative group of $d\times d$ orthogonal matrices.
We want to test  the null hypothesis 
\[
H_0:\quad  PT^{-1}=P \quad \text{for all~}T \in G.
\]

Assume that $G$ is generated by $k$ elements, say $\{T_1,\dots,T_k\}$.
Then, as we have seen in \S\ref{S:generators}, 
it suffices just to test whether $PT_j^{-1}=P$ for $j=1,2,\dots, k$.

We now proceed as follows.
Let $\aleph_n:=\{\bX_1, \ldots, \bX_n\}$ be an i.i.d.\ sample of vectors in $\RR^d$ with common distribution $P$.
Consider the following testing procedure. 	
\begin{enumerate}
\item Split the sample $\aleph_n$ at random into two disjoint subsamples
\[
\aleph^{n_1}:=\{\bX_1, \ldots \bX_{n1}\},
\quad
\aleph^{n_2}= \{\bX'_1, \ldots \bX'_{n2}\},
\]
providing two independent samples, which we will take typically of equal sizes.

\item Generate $k$ independent random directions $h_1, \ldots h_k$,
 uniformly distributed on the unit sphere of $\RR^d$. 
		
\item Given a unit vector $h\in\RR^d$, 
denote by $F_h$ the one-dimensional distribution of $\langle h,\bX\rangle$, 
and by $F_h^{T_j}$  that of $\langle h,T_j(\bX)\rangle$ for $j=1, \ldots, k$. 
The previous results suggest testing whether $F_{h_j}^{T_j}=F_{h_j}$, for $j=1, \ldots k$.  

\item Next,  we perform  $k$ Kolmogorov--Smirnov two-sample tests between 
\[
\langle h_j,\aleph^{n_1}\rangle
\quad\text{and}\quad  
\langle h_j,\aleph^{n_2}_{T_{j}}\rangle, \quad j=1, \ldots k,
\]
each one at level $\alpha/k$, where 
\[
\langle h, \aleph^{n_1}\rangle:=\Bigl\{\langle h,\bX_1\rangle, \ldots, \langle h,\bX_{n_1}\rangle \Bigr\}
\]
and
\[
\langle h, \aleph^{n_2}_{T}\rangle:=\Bigl\{\langle h,T(\bX'_1)\rangle, \ldots, \langle h,T(\bX'_{n_2})\rangle\Bigr\}.
\]
		
\item Given $h, T$, let $F_{n_1,h} ^{T}$ and $F_{n_2,h}^{T}$ be the empirical distributions of 
$\langle h, \aleph^{n_1}\rangle$ and $\langle h, \aleph^{n_2}_{T}\rangle$ respectively.
The Kolmogorov--Smirnov statistic for our problem is given by
\[
KS(n_1,n_2, h, T):= \sup_{t \in \RR} \bigl| F_{n_1,h}^{T}(t) -  F_{n_2,h}^{T}(t)\bigr|.
\]

\item Given $\alpha$ with $0<\alpha<1/k$, choose $c_{\alpha/k}$ such that
\[
P\Bigl(KS(n_1, n_2, h_j,T_j)> c_{\alpha/k}\Bigr)= \alpha/k, 
\]
Observe that $c_{\alpha/2, n_1,n_2}$ does not depend on~$j$.
This is because 
\[
\langle h_j,T_j(\bX)\rangle=\langle T_j^*(h_j),\bX\rangle,
\] 
and all the vectors $T_j^*(h_j)$ have the same distribution, namely the uniform distribution on the unit 
sphere in $\RR^d$ (it is here that we need the assumption that the $T_j$ are orthogonal matrices).
		
\item Reject the null hypothesis if, for at least one of the $k$ Kolmogorov--Smirnov tests, the null hypotheses is rejected.   
\end{enumerate} 
	
\begin{theorem} 
Let $P$ be a Borel probability measure on $\RR^d$, where $d\ge2$.
Assume that:
\begin{enumerate}[\normalfont(i)]
\item the absolute moments $m_N:=\int\|x\|^N\,dP(x)$ are all finite and satisfy 
the condition $\sum_{N\ge1}m_N^{-1/N}=\infty$;
\item $P$ is absolutely continuous with respect to Lebesgue measure on $\RR^d$.
\end{enumerate}
Then the preceding test  has a level between $\alpha/k$ and $\alpha$ under the null hypothesis~$H_0$, 
and is consistent under the alternative. 
\end{theorem}

\begin{proof}
By the choice of $c_{\alpha/k}$ and Bonferroni's inequality, we have
\[
\alpha/k \le P\Bigl(\bigcup_{j=1}^k\{KS(n_1, n_2, h_j, T_j) > c_{\alpha/k}\} \Bigr ) \le \alpha.
\]
Thus the test has a level between $\alpha/k$ and $\alpha$.
	
The proof of consistency follows the same lines as the one given  in \cite[Theorem~3.1]{CFR06}.		
Under the alternative hypothesis,
there exists $j_0\in\{1,\dots,k\}$ such that   $P \ne PT^{-1}_{j_0}$.  
Then, by Theorem~\ref{T:CWgen},
for almost all $h\in\RR^d$, there exists $t_h\in\RR$ such that 
$\delta_h:=|F_{h}^{T_{j_0}}(t_h) - F_{h}(t_h)| >0$. 
As $P$ is absolutely continuous with respect to Lebesgue measure on $\RR^d$, 
both $F_h$ and $F_h^{T_{j_0}}$ are continuous, 
and so,  by the strong law of large numbers, 
\[
\lim_{n\to\infty}F_{n,h}(t_h) = F_{h}(t_h)
\quad\text{and}\quad
\lim_{n\to\infty}F_{n,h}^{T_{j_0}}(t_h) = F_{h}^{T_{j_0}}(t_h) \quad\text{a.s.}
\]
Hence, by the triangle inequality,
\[
\liminf_{n_1,n_2\to\infty}KS(n_1,n_2,h,T_{j_0})\ge \delta_h >0\quad\text{a.s.}
\]
This establishes consistency.
\end{proof}


\subsection{Testing using more than one random projection} \label{SS:testing2}

The algorithm presented in \S\ref{SS:testing1} is consistent and distribution-free. 
In this section we consider the case where we use more than one projection, 
as suggested in \cite{CFR07},
in order to increase the power of the test.
Moreover, we no longer need to assume that the matrices $T_j$ are orthogonal. 
The price to be paid is that the proposed statistic is no longer distribution-free. 

In this case we do not need to split the sample in two subsamples.
It just consists on taking $\ell$ random directions $h_1, \ldots, h_\ell$, 
for each direction $h_i$ calculating the statistic $KS(n, h_i, T_j)$ for the univariate projected data, 
and then taking the maximum of them:
$$
D_{n,\ell}:=
\max_{\substack{i=1, \ldots, \ell\\j=1,\dots,k}}KS(n, h_i, T_j)
=\max_{\substack{i=1, \ldots, \ell\\j=1,\dots,k}}\sup_{t \in \mathbb R} | F_{n,h_i}(t) -  F_{n,h_i}^{T_j} (t)|.
$$
Here $F_{n,h_i}(t)$ and $F^{T_j}_{n,h_i}(t)$ are the empirical distribution of 
\[
\langle h_i,\aleph^{n}\rangle
:=\{\langle h_i,\bX_1\rangle, \ldots, \langle h_i,\bX_{n}\rangle \}
\]
and 
\[
\langle h_i,\aleph^{n}_{T_j}\rangle
:=\{\langle h_i,T_j(\bX_{1})\rangle , \ldots, \langle h_i,T_j(\bX_{n})\rangle \},
\]
for $ i=1, \ldots \ell$ and  $j=1, \ldots,k$, where $k$ is the number of generators of the group.

Since  the statistic  is no longer distribution-free for $\ell\ge 2$, 
in order to obtain the critical value for a level-$\alpha$ test, 
we approximate the distribution using bootstrap on the original sample $\bX_1, \ldots, \bX_n$
by generating a large enough number $B$ of values of $D_{n,\ell}$ for each bootstrap sample.
More precisely, for $r=1, \ldots, B$ repeat: 
\begin{enumerate}
	\item Generate a bootstrap sample of $\mathbf X_1^*, \ldots, \mathbf X_n^*$,  by random sampling with replacement from $\aleph_n$, and generate $(\ell+k)$ i.i.d.\ random directions $h_1, \ldots, h_{\ell+k}$. 
	\item Calculate $D_{n,\ell}^*$ based on $\mathbf X_1^*, \ldots, \mathbf X_n^*$ and $T_j(\mathbf X_1^*), \ldots, T_j(\mathbf X_n^*)$. 
\end{enumerate}
We end up with a sample  $ \cD^*:= \{D_{n,l}^{*1}, \ldots , D_{n,l}^{*B}\}$ of size $B$, 
and take as critical value the $(1-\alpha)$-quantile of the empirical distribution of $\cD^*$.  
The validity of the bootstrap in this case  follows from  \cite[Theorems~3 and~4]{Pr95}.


\section{Application to exchangeability}\label{S:exch}

\subsection{Background}\label{SS:exchback}

A $d$-tuple of random variables $\bX:=(X_1,\dots,X_d)$ is said to be \emph{exchangeable} if
it has the same distribution as $(X_{\sigma(1)},\dots,X_{\sigma(d)})$
for every permutation $\sigma$ of $\{1,\dots,d\}$.
This is equivalent to demanding that the distribution $P$ of $\bX$ be $G$-invariant, 
where $G$ is the group of $d\times d$ permutation matrices.

The notion of exchangeability plays a central role in probability and statistics. 
See for instance \cite{Al85} for a detailed study of exchangeability, 
and \cite{Ki78} on some of the important uses of it. 
In particular, permutation tests provide exact level tests for a wide variety of practical testing problems, 
but they rely on the assumption of exchangeability.
 
From the point of view of applications, 
it is quite common to assume that 
a given vector of data $(X_1, \ldots, X_n)$
(where the information $X_i$ corresponds to the $i$-th subject of a study)
does not depend on the order in which the data are collected. 
Although independence of the $X_i$ is quite often assumed,
this is stronger than the notion of exchangeability.
The relationship is made precise by de Finetti's theorem,
according to which, if $\bX:= (X_1, X_2, \ldots)$
is a sequence of random variables 
whose distribution is invariant under permutations of finitely many coordinates,
then the law of $\bX$ is a mixture of i.i.d.\ measures.  
De Finetti's theorem applies to an infinite exchangeable sequence of random variables,
but there are also versions that apply to finite exchangeable sequences,
see e.g.\ \cite{DF80, JKY16, KY19}  and the references therein.

The problem of testing for exchangeability has been considered quite extensively in the literature.
Modarres \cite{Mo08} compares a number of different methods: the run test, the nearest neighbour test,
the rank test, the sign test and the bootstrap test.
The problem of testing exchangeability in an on-line mode has been considered  by 
Vovk, Nouretdinov and Gammerman \cite{VNG03}.
The data are observed sequentially,
and the procedure provides a way to monitor on-line the evidence against the exchangeability assumption, 
based on exchangeability martingales. 
A test for the case of binary sequences can be found in \cite{RRLK21}, using a similar approach. 

Genest, Ne\u slehov\'a and Quessy \cite{GNQ12} propose an interesting test for a closely related problem: exchangeability of copulas for two-dimensional data. This has been  extended to arbitrary finite dimensions by Harder and Stadtm\"uller \cite{HS17},
using the observation that $\Sigma_d$ can be generated by just two elements (see Example~\ref{Ex:2gen} above). See also the recent article \cite{BQ21} by  Bahraoui and Quessy, where they also consider the problem of testing if the copula is exchangeable, but, instead of using empirical copulas as in the previous results, they consider a test based on the copula characteristic function. As they mention:
``From a modeling point-of-view, it may be of interest to check whether the copula of a population is exchangeable. The main point here is that this step may be accomplished without making any assumption on the marginal distributions. Hence, a vector can have exchangeable components at the level of its dependence structure, that is, its copula, whether or not its marginal distributions are identical." Under the extra assumption that all marginals have the same distribution, it provides a test for exchangeability of the distributions.
However, they need to use bootstrap in order to perform their test.
	 
In the next paragraph, we  perform some simulations to show how the test procedures described in \S\ref{S:testing}
apply to this situation. In \S\ref{SS:exchsim2} we compare our test for exchangeability with the one proposed by Harder and Stadtm\"uller  in \cite{HS17}.


\subsection{Simulations for the test for exchangeability}\label{SS:exchsim1}
We consider a sample of a multivariate normal distribution in dimension $d$, with mean zero and covariance matrix given by a mixture
$(1-\xi)\Sigma^{(1)} +\xi \Sigma^{(2)},$ where $\Sigma^{(1)}$ corresponds to a exchangeable distribution and $\Sigma^{(2)}$ is a Toeplitz matrix, given by 
\[
\Sigma^{(1)}:= 
\begin{pmatrix}
1 & \rho & \cdots & \rho\\
\rho& 1 & \cdots & \rho\\
\vdots & \vdots & \ddots & \vdots\\
\rho& \rho & \cdots & 1
\end{pmatrix},  
\quad  
\Sigma^{(2)}:= 
\begin{pmatrix}
1 & \rho & \rho^2 & \cdots & \rho^{d-1}\\
\rho & 1 & \rho &  \cdots & \rho^{d-2}\\
\rho^2 & \rho &  1 &   \cdots & \rho^{d-3}\\
\vdots & \vdots & \ddots  & \ddots & \vdots\\
\rho^{d-1}& \rho^{d-2} & \cdots &\rho & 1
\end{pmatrix}.
\]
In this simulation, we take $\rho:=1/2$. If $\xi=0$, then the distribution is exchangeable, 
and as it increases to $1$ we get further from the null hypothesis. 
To analyze the power function of the proposed test, 
we consider some different scenarios:
the dimension $d$ being $6$ and $10$, 
the sample size $n$ being $1000$ and $5000$, 
and the number of projections being $1,10,50, 100$. 
In all cases, the level of the test is  $0.05$.
In Figure~\ref{F:normal} we plot the empirical power function as a function $g(\xi)$ of $\xi \in [0,1]$.

\begin{figure}[ht]
\centering
\subfloat[]{\includegraphics[scale=0.65]{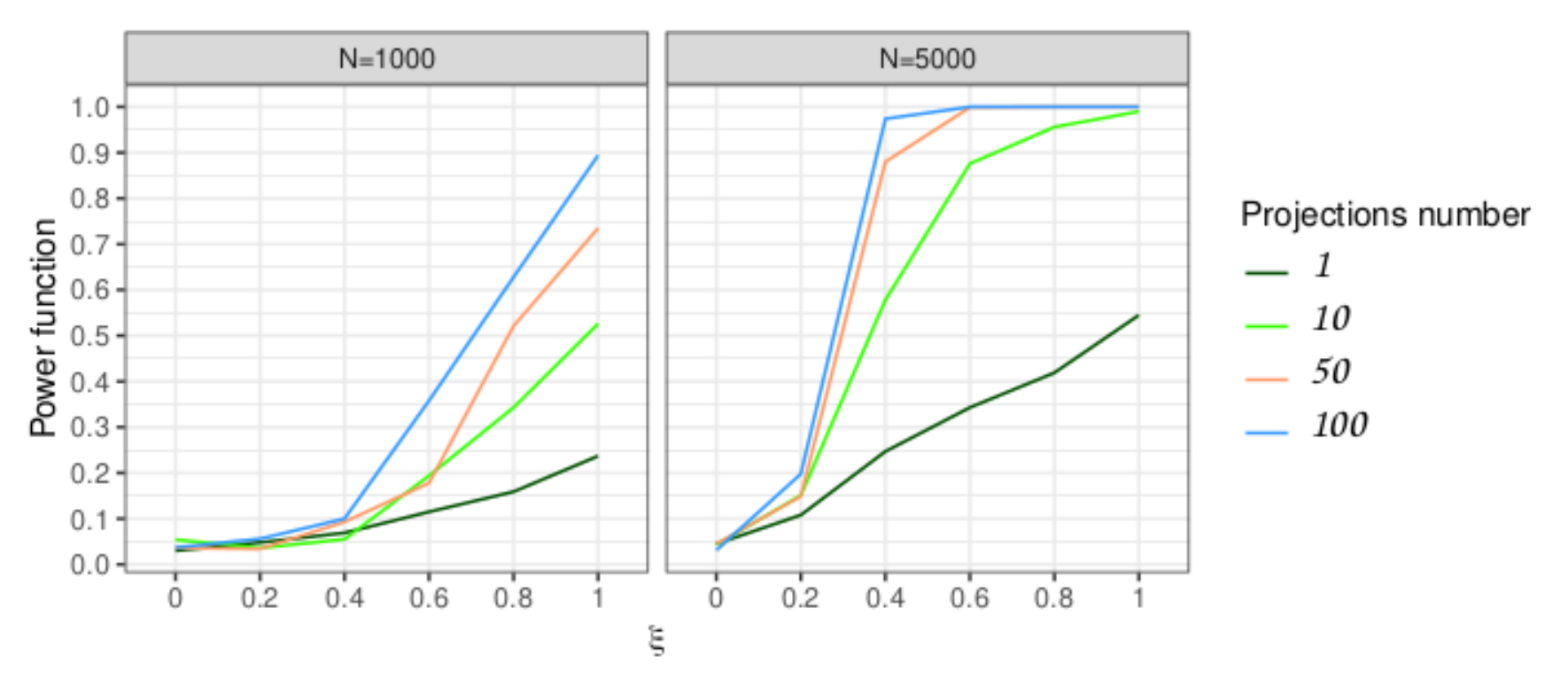}}
\hfill
\subfloat[]{\includegraphics[scale=0.65]{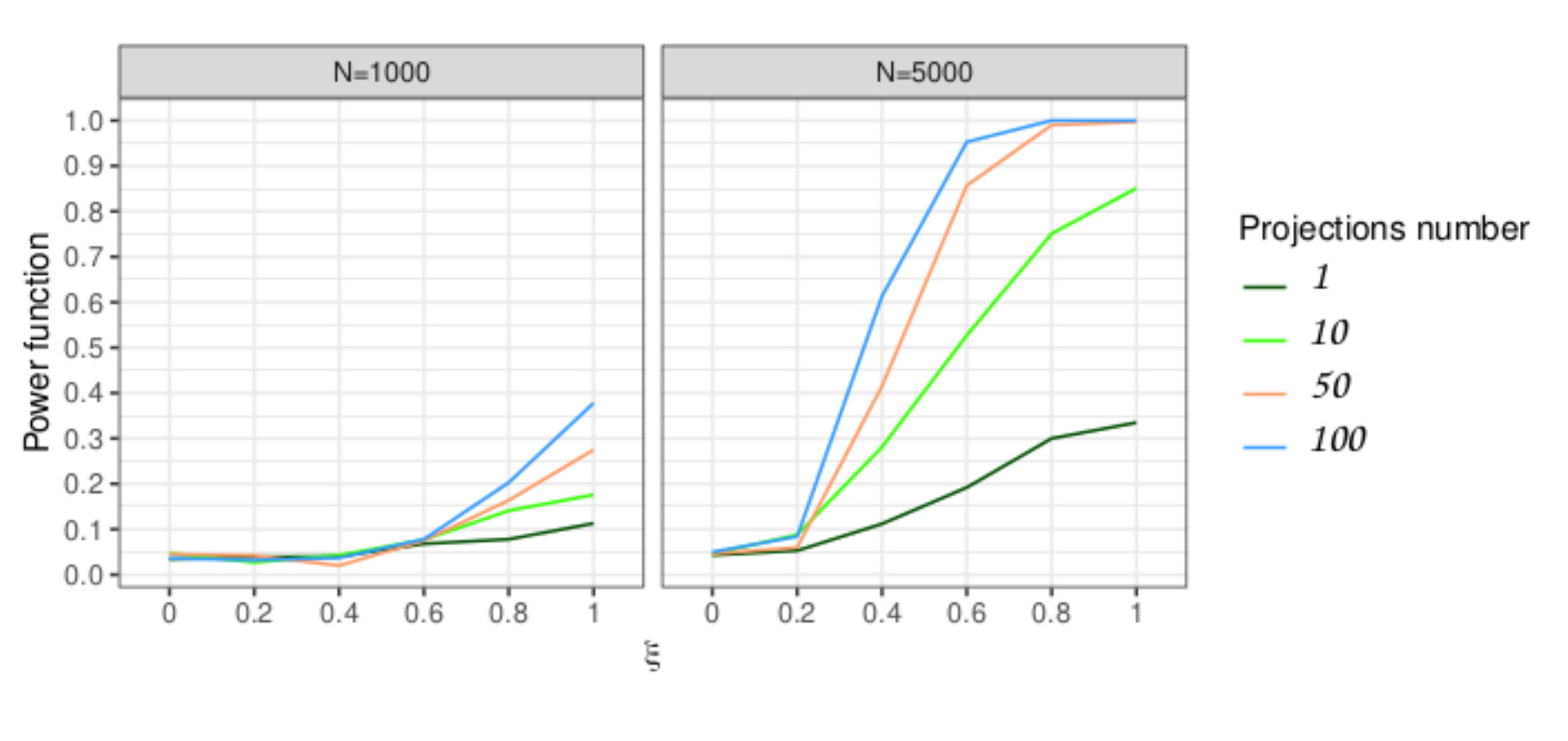}}
\caption{\small Empirical power function $g(\xi)$ for $\xi \in [0,1]$. Upper panel (a) for  dimension $d=6$, lower panel (b) for $d=10$.}\label{F:normal}
\end{figure}


\subsection{Comparison with a copula-based test}\label{SS:exchsim2}
As we mentioned in \S\ref{SS:exchback},
Harder and Stadtm\"uller \cite{HS17} proposed an interesting test for exchangeability of copulas, 
which we describe briefly here. Assuming that all marginal distributions are equal it provides a exchangeability test for distributions. However, without the extra assumption that all marginals coincide, the exchangeability of the copula does not imply exchangeability of the probability distributions. 
Thus, without this assumption,  it is also necessary to test that all marginal distributions coincide, which, for even moderate dimensions, makes the problem harder. For the comparison below, we consider a case where all marginal distributions coincide. 

Let $\bX:=(X_1, \ldots, X_d)$ be a random vector with joint distribution $F$ and 
continuous marginal distributions $F_1 \dots F_d$.
Assume that $F_1,\dots,F_d$ are continuous.
Then $U_j:=F_j(X_j)$ is uniformly distributed on the interval $[0,1]$ for each~$j$,
and the distribution $C$ of $(U_1, \ldots, U_d)$ is the unique copula 
such that $F(x_1,\dots,x_d)= C(F_1(x_1), \ldots, F_d(x_d))$.

Given an i.i.d. sample $\bX_1, \ldots, \bX_n \in \RR^d$, set
\begin{equation}\label{E:empcop}
\hat C_n(\bu):= \frac{1}{n} \sum_{i=1}^n \prod_{j=1}^d \cI_{\{\hat F_{jn}(X_{ij})\leq u_j\}}, 
\quad \bu:=(u_1, \ldots, u_d) \in [0,1]^d,  
\end{equation}
where $\hat F_{jn}, \ j=1, \ldots, d$ are the empirical marginal distributions of the sample.
Among many other interesting asymptotic results,  Deheuvels (\cite{De81}) showed that
\begin{equation}\label{E:deh}
\sup_{\bu \in [0,1]^d}|\hat C_n(\bu) - C(\bu)\vert = O\bigl(n^{-\frac{1}{2}}(\log\log n)^\frac{1}{2}\bigr) \quad\text{a.s.},
\end{equation}
which suggests using this statistic for testing.

Based on $\hat C_n(\bu)$, Harder and Stadtm\"uller \cite{HS17}
proposed a test of exchangeability for the problem
\begin{align*}
&H_0:  \quad  C(\bu)= C(\bu_{\sigma}) \quad\text{for all $\bu \in [0,1]^d$ and all $\sigma \in \Sigma_d$}, \\
\intertext{against}
&H_A:\quad C(\bu)\ne  C(\bu_{\sigma}) \quad\text{for some $\bu \in [0,1]^d$ and some $\sigma \in \Sigma_d$},
\end{align*}
by performing a test based on statistics defined as integral versions of the difference 
$(\hat C_n(\bu)- \hat C_n(\bu_{\sigma}))$, such as
\begin{equation}\label{E:copulas}
S_n:= \sum_{\sigma \in \cG_0} \int_{[0,1]^d}\bigl(\hat C_n(\bu)- \hat C_n(\bu_{\sigma})\bigr)^2 w(\bu, \sigma)\,dC(\bu), 
\end{equation}
where $\cG_0$ is a set of generators for the permutation group $\Sigma_d$,
and $w(\bu, \sigma)$ is a bounded continuous weight function. 
From equation~\eqref{E:deh} and the dominated convergence theorem, it follows  that 
$S_n \to S$ a.s., where 
\[
S:= \sum_{\sigma \in \cG_0} \int_{[0,1]^d}\bigl(C(\bu)-  C(\bu_{\sigma})\bigr)^2 w(\bu, \sigma)\,dC(u).
\]
This is shown in \cite[ Lemma 3.1]{HS17}, 
while the asymptotic distribution is derived in  \cite[Theorem 3.4]{HS17} under some regularity assumptions.

As in \cite{HS17}, we consider hierarchical copulas for two different scenarios:
\[
C^1_{\theta_0, \theta_1}(u):= C_{\theta_0} \bigl( u_1, C_{\theta_1}(u_2, u_3) \bigr),
\]
and
\[
C^2_{\theta_0, \theta_1}(u):= C_{\theta_0} \bigl( C_{\theta_1}(u_1, u_2) , u_3\bigr),
\] 
where $C_{\theta}$ is the Clayton bivariate copula with parameter  $\theta$.
The parameters  $\theta_0$ and $\theta_1$ are chosen so that their Kendall indices $\tau$ are   
\[
\tau_0:=5/12-\xi/60 \quad\text{and}\quad \tau_1:=5/12+\xi/60 \quad\text{for~}\xi \in \{0,1,\ldots,7\}.
\]
When $\xi=0$ we are under the null hypothesis, and as $\xi$ increases we are further away from the null.
The number of random directions and sample sizes are the same as in \S\ref{SS:exchsim1}.  
 
The empirical power functions are shown in Figures~\ref{F:copula1} and~\ref{F:copula2}.

\begin{figure}[ht]
\centering
\includegraphics[scale=0.75]{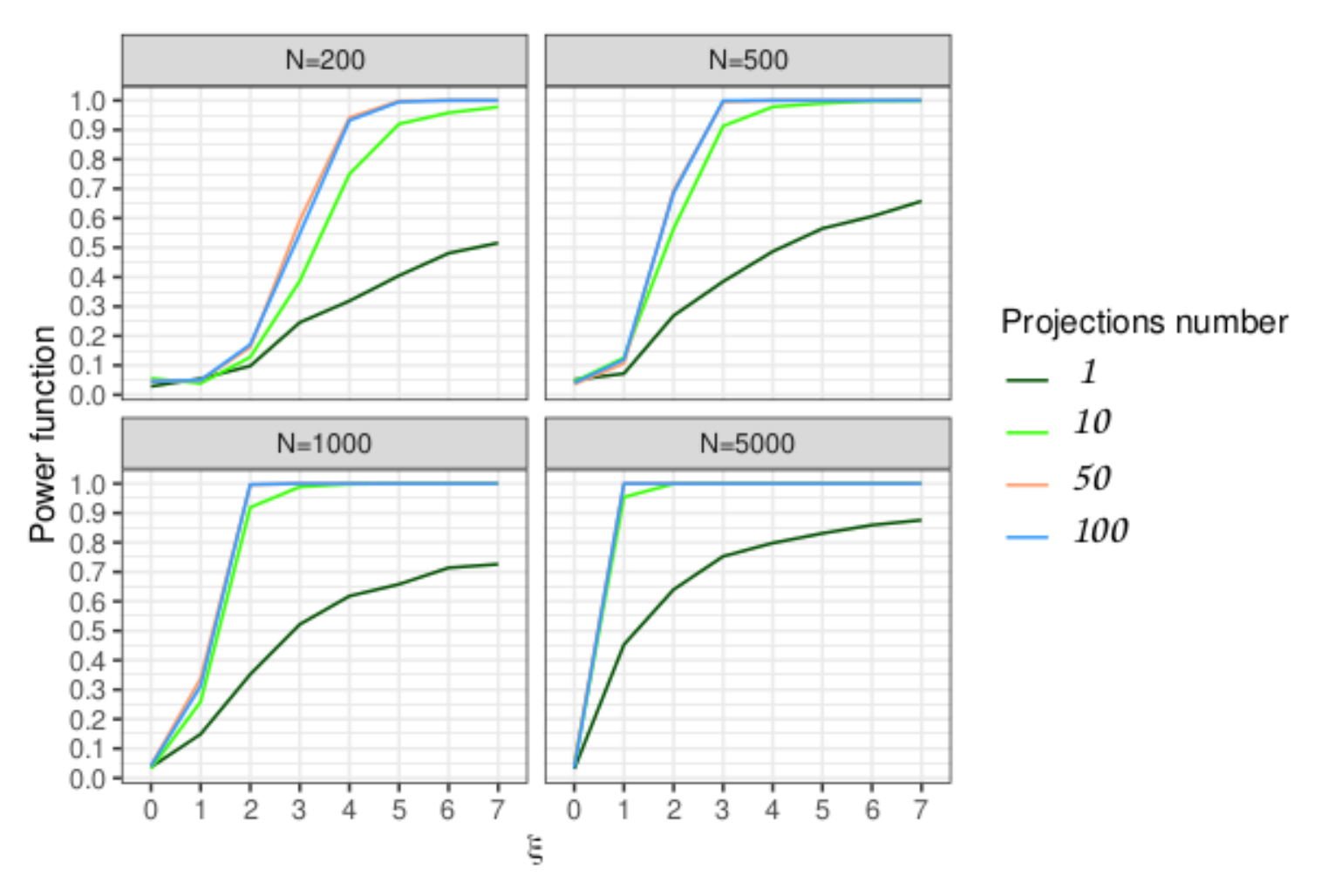}
\caption{\small Empirical power function for the hierarchical copula for the exchangeability test as a function of  $\xi \in \{0,1,\ldots,7\}$, in dimension 3  for the copula  $C^1$.}
\label{F:copula1} 
\end{figure}

\begin{figure}[ht]
\centering
\includegraphics[scale=0.75]{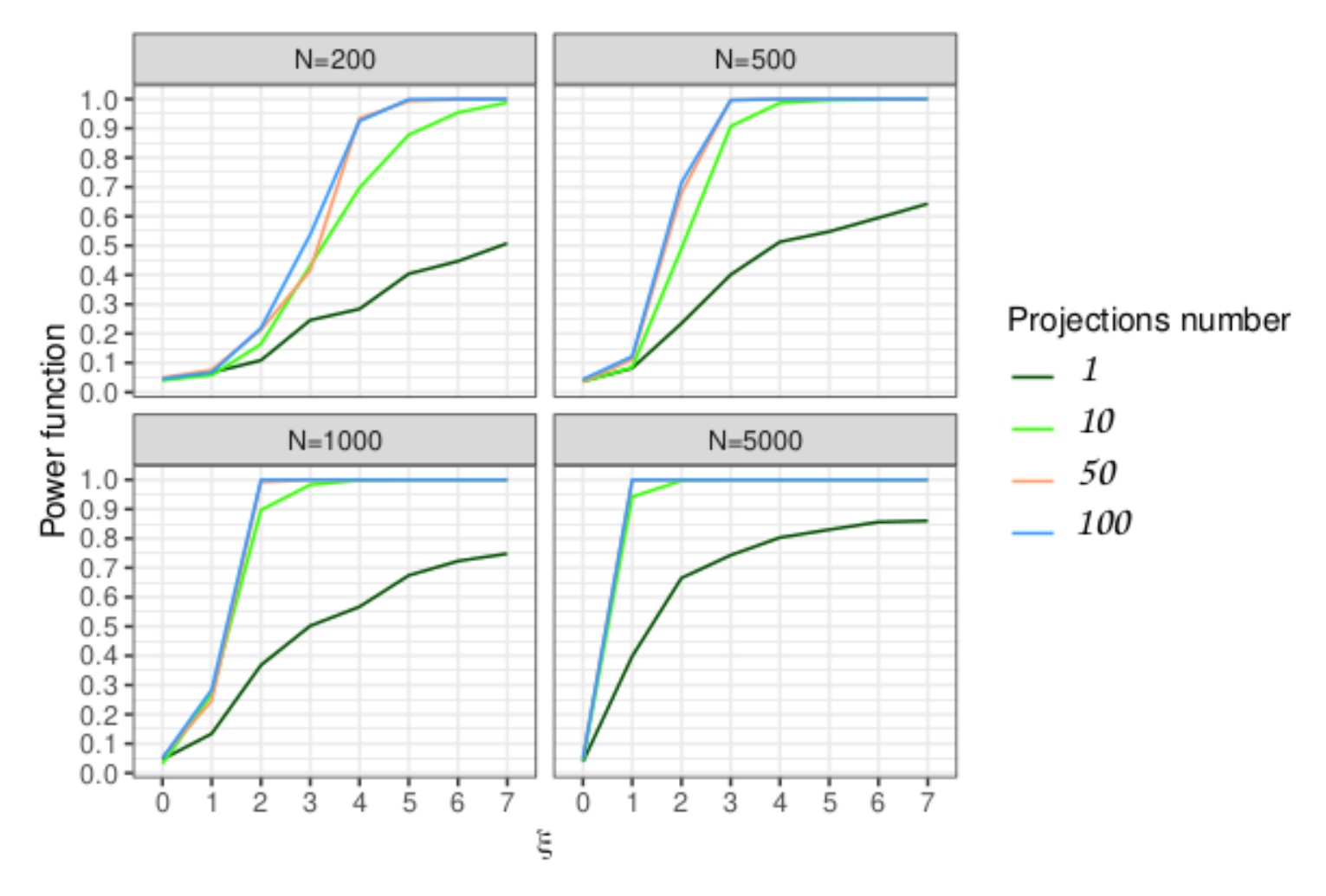}
\caption{\small Empirical power function for the hierarchical copula for the exchangeability test as a function of  $\xi \in \{0,1,\ldots,7\}$, in dimension 3  for the copula  $C^2$.}
\label{F:copula2} 
\end{figure}

Our results  show that the empirical power functions are better than those reported in \cite{HS17} for a sample size $N=1000$. We compare with the best results obtained for this scenario in \cite{HS17} using the statistic $S_{1000}$ defined by equation \eqref{E:copulas} with our proposal $D_{1000,50}$, for different values of $\xi$, 
see Table~\ref{Ta:pot}.

{\begin{table}[ht]
\caption{\small Empirical power functions for the tests based on  $S_{n}$  and $D_{n,50}$ for a sample of size $1000$. 
We consider a random vector with uniform marginals 
and the copulas  $C^1_{\theta_0, \theta_1}$ and $C^2_{\theta_0, \theta_1}$ in dimension $3$, 
for different values of $\xi$ when the level of the test is  $5 \%$.}
\label{Ta:pot}
\begin{center}
\begin{tabular}{ccccc}\toprule
& \multicolumn{2}{c}{$C^1_{\theta_0, \theta_1}$} & \multicolumn{2}{c}{$C^2_{\theta_0, \theta_1}$} 
\\\cmidrule(lr){2-3}\cmidrule(lr){4-5}
$\xi$ & $S_{1000}$ & $D_{1000,50}$  & $S_{1000}$  & $D_{1000,50}$   \\\midrule
0   & 0.044 & 0.043& 0.033 &  0.052 \\
1& 0.182& 	0.338 & 0.112 & 0.246 \\
2 & 0.667 & 0.996 & 0.416 & 0.992 \\
3   & 0.981 &1.000 & 0.865 &  1.000\\
4 & 1.000 & 1.000 & 0.999 & 1.000\\
5   & 1.000 & 1.000 & 1.000 & 1.000\\
\bottomrule
\end{tabular}
\end{center}
\end{table}}

\section{Sign-invariant exchangeability}\label{S:signexch}

\subsection{Background}\label{SS:signexchback}
Berman \cite{Be62,Be65} introduced the notion of sign-invariant exchangeability, 
which is defined as follows.
A $d$-tuple of random variables $\bX:= (X_1, \ldots, X_d)$ is \emph{sign-invariant} if
it has the same distribution as $(\epsilon_1 X_1, \ldots, \epsilon_d X_d)$ 
for every choice of $(\epsilon_1,\dots,\epsilon_d) \in \{-1,1\}^d$. 
It is \emph{sign-invariant exchangeable} if it is both sign-invariant and exchangeable.

Equivalently, $(X_1, \ldots, X_d)$ is sign-invariant exchangeable if and only if it has the same
distribution as 
$(\epsilon_1 X_{\sigma(1)},\dots,\epsilon_d X_{\sigma(d)})$ for all $\sigma\in\Sigma_d$
and  $(\epsilon_1,\dots,\epsilon_d) \in \{-1,1\}^d$.
This amounts to saying that the distribution $P$ of $\bX$ is $G$-invariant,
where $G$ is the group of $d\times d$ signed permutation matrices.
As remarked in  Example~\ref{Ex:sign}, $G$ can be generated by three matrices $T_1,T_2,T_3$,
so, to test for $G$-invariance, it suffices to test whether $P=PT^{-1}_j$ for $j=1,2,3$.


\subsection{Simulations for the test for sign-invariant exchangeability}\label{SS:signexchsim}
  
We consider a sample of a multivariate  normal distribution in dimension $d$ for $d=3, 6, 10$, 
with mean zero and covariance matrix given by  $\Sigma^{(1)}$, defined in \S\ref{SS:exchsim1}
(where now $\rho$ is variable).
When $\rho=0$ the distribution is  sign-invariant exchangeable. 
We consider sample sizes of $200, 500, 1000$ and $5000$, 
and we consider  $1, 10, 50$ and $100$ random directions in $\RR^3$. 

A plot of the empirical power function of the test as a function of $\rho \in [0,1]$
using \S\ref{SS:signexchback} and \S\ref{SS:testing2} is given in Figure~\ref{F:signexch}.

\begin{figure}[ht]
\centering
\includegraphics[width=0.95\textwidth]{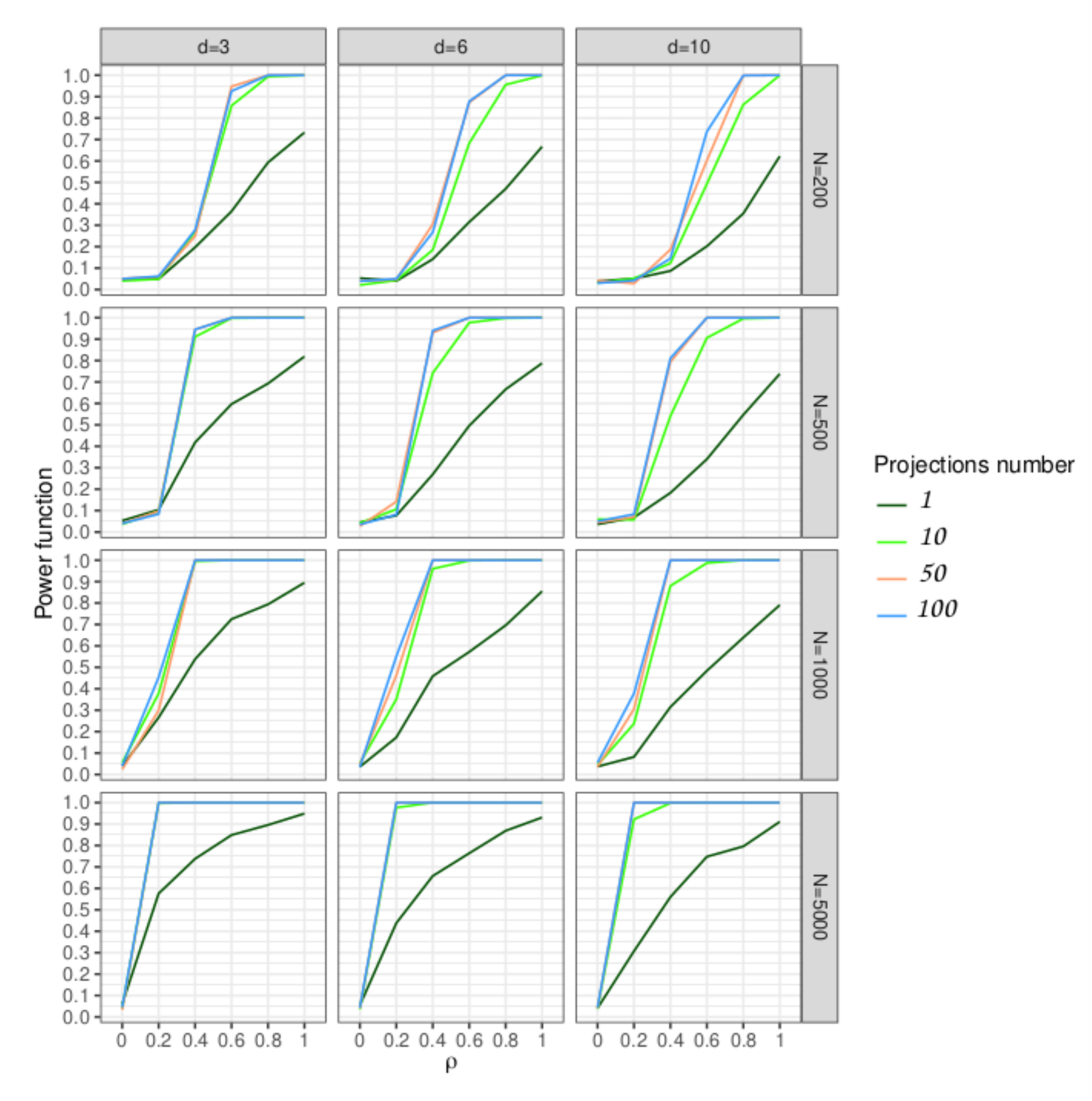}
\caption{\small Empirical power function for sign-invariant exchangeable testing  with $\alpha=0.05$ and varying $\rho \in [0,1]$. 
We consider samples sizes of $200,500,1000$ and $5000$, and   $1,10,50$ and $100$ random directions in $\RR^2$.}
\label{F:signexch} 
\end{figure}

\section{An infinite-dimensional example}\label{S:infdim}

\subsection{Background}
As mentioned in \S\ref{S:CW},
there are several variants of the Cram\'er--Wold theorem, 
including versions for Hilbert spaces \cite{CFR07}
and even Banach spaces \cite{CF09}.
These can be exploited to provide tests for 
$G$-invariance  in infinite-dimensional spaces.
We consider here the case of a Hilbert space.

Let $\cH$ be a separable infinite-dimensional Hilbert space,
let $P$ be a Borel probability measure on $\cH$,
and let $G$ be a group of invertible continuous linear
maps $T:\cH\to \cH$. We want to test the null hypothesis
\[
H_0: \quad PT^{-1}=P \quad\text{for all~}T\in G.
\]
Nearly everything goes through as before.
The only adjustment needed is in Theorem~\ref{T:CWgen},
because Lebesgue measure no longer makes sense 
in infinite dimensions. It role is taken by a non-degenerate gaussian
measure on $\cH$. What this means is explained in detail in \cite[\S4]{CFR07}, where one can find the proof of the following result,
which serves as a replacement for Theorem~\ref{T:CWgen}.

\begin{theorem}[\protect{\cite[Theorem~4.1]{CFR07}}]\label{T:cwgauss}
Let $\cH$ be a separable Hilbert space,
and let $\mu$ be a non-degenerate gaussian measure on $\cH$.
Let $P,Q$ be  Borel probability measures on $\cH$.
Assume that
the absolute moments $m_N:=\int\|x\|^N\,dP(x)$ are all finite and satisfy 
\begin{equation}
\sum_{N\ge1}m_N^{-1/N}=\infty.
\end{equation}
If the set $\cE(P,Q)$ is of positive  $\mu$-measure in $\cH$,
then $P=Q$.
\end{theorem}


\subsection{Simulations for an infinite-dimensional example}

In this example, we perform a simulation study on the exchangeability test
for multidimensional functional data in the Hilbert space $\mathcal{L}^3:=\otimes_{i=1}^3 L^2[0,1]$. 
We consider a sample of i.i.d.\ random functional vectors  $\{\mathbf{X}_i\}_{i=1, \ldots,n} \subset \mathcal{L}^3$,
 where the marginal components are defined as 
\[
X_{i,j}(t):= m(t)+ \epsilon_{i,j}(t),  \quad i \in \{1,2,\dots, n\} \textrm{\and\ } j \in \{1,2,3\},
\]
where $m(t):= \cos (2 \pi t)$, and where  $\epsilon_{1,j}, \ldots, \epsilon_{n,j}$ are i.i.d\ gaussian processes with
\[
\textrm{Cov}(\epsilon_{i,j}(t), \epsilon_{i,j}(s))= \textrm{exp} \left( - \vert s-t \vert \right).
\]
Finally, we assume that the correlation function between the marginal components satisfies
\[
\textrm{Cor}(\epsilon_{i,j_1}(t), \epsilon_{i,j_2}(t))=
\begin{cases}
0.5 -\delta, &\text{if~} j_1=1,~j_2=3,\\
0.5, &\text{if~} j_1=2,~j_2=3,\\
0.5+\delta &\text{if~} j_1=1,~j_2=2.
\end{cases}
\] 
For $\delta=0$, the functional vector is exchangeable.  

The random projections are taken over a standard multivariate brownian $\mathbf{W}$ in $[0,1]^3$, that is,
$$\langle \mathbf{X}, \mathbf{W} \rangle = \sum_{i=1}^3 \langle {X}_j, {W}_j \rangle, $$
where  ${W}_j$ are the i.i.d.\  standard brownian in $[0,1]$ and $$\langle {X}_j, {W}_j \rangle= \int_0^1 {X}_j(t){W}_j(t) dt.$$
The functions are discretized on an equispaced grid of size $100$ in $[0,1]$.
Figure~\ref{functional} depicts a realization of the functional random vector for $\delta=0.3$. 

\begin{figure}[ht]
\centering
 \subfloat{\includegraphics[width=125mm]{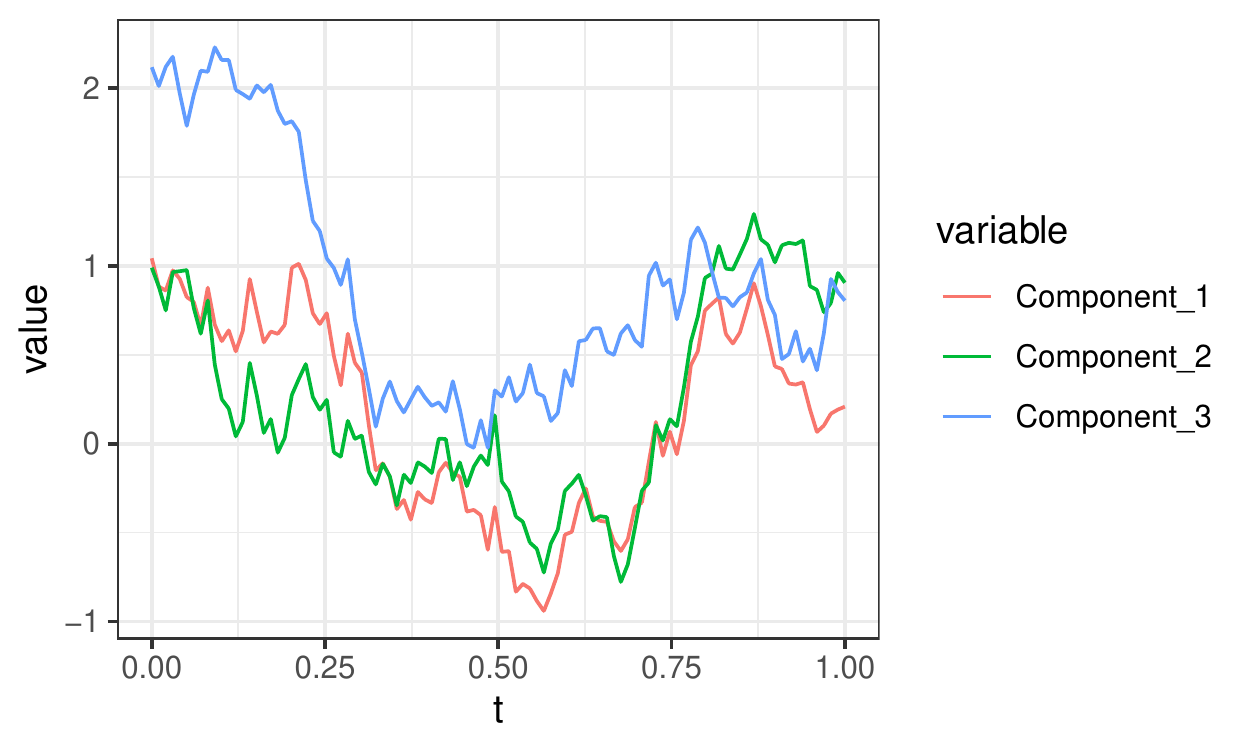}}
\caption{Marginal curves of a simulation of $\mathbf{X}$ for $\delta=0.3$}
\label{functional}
\end{figure}

The statistic $D_{n,\ell}$ is calculated (as in Section~\ref{SS:testing2}) for the values  $n \in \{250,500,1000\}$ 
and $\ell \in \{10,50,100\}$. The power functions through $1000$ replicates are determined in
Table~\ref{Tab:pvalores}. We find that the results are quite good.

\begin{table}
\caption{Empirical power function $D_{n,\ell}$ for $\delta \in [0,0.3]$,  $n \in \{250,500,1000\}$ and $\ell \in \{10,50,100\}$.  Significance level= $0.05$. } \label{Tab:pvalores}
	\begin{center}
\begin{tabular}{c|ccc|ccc|ccc}
 \toprule[0.4 mm]
 & \multicolumn{3}{c|}{$n=250$}  &\multicolumn{3}{c|}{$n=500$}  &\multicolumn{3}{c}{$n=1000$} \\
\midrule[0.4 mm]
$\delta \, \, \backslash  \,\, \ell$  \ & $10$ & $50$& $100$&$10$ & $50$& $100$&$10$ & $50$& $100$ \\
\hline
$0$ & 0.044 & 0.035 &0.051& 0.044&  0.040 & 0.045 &0.051 & 0.04 & 0.044 \\
$0.05$ & 0.050& 0.045  &0.059& 0.043 & 0.040 & 0.061 &0.078 & 0.065 & 0.080  \\
$0.1$ & 0.078  & 0.069&0.089& 0.116  & 0.103 & 0.152 &0.264& 0.266&  0.350\\
$0.15$ & 0.116 & 0.121&0.194& 0.318 & 0.342 & 0.474  &0.727 &0.877 & 0.946 \\
$0.20$ & 0.298 & 0.331&0.470& 0.643  & 0.842& 0.945 &0.926 & 0.999&  1.000\\
$0.25$& 0.576 & 0.721 &0.853& 0.851 & 0.995 & 1.00 & 0.977 & 1.000&  1.000\\
$0.30$ & 0.781 & 0.977 &0.997& 0.989 & 0.995& 1.00 &0.996 & 1.000 &  1.000\\
\toprule[0.4 mm]
\end{tabular}
\end{center}
\end{table}


\section{Real-data examples}

We conclude the article with two examples drawn from real datasets.
In both cases, the test is for exchangeability.

\subsection{Biometric measurements}

These data were collected in a study of how data on various characteristics of the blood varied with sport, body size and sex of the athlete, by the Australian Institute of Sport. The data set, see  \cite{CW94}, is available in the package \textit{locfit} of the $R$ software. The data set  is composed of $202$ observations and $13$ variables. As in \cite{BQ21}, five covariates are considered: red cell count (RCC),  haematocrit (Hc), haemoglobin (Hg),  lean body mass (LBM) and height (Ht). It is clear that the marginal distributions are different, so the distribution is not exchangeable. In order to check the exchangeability of the copula, the marginal distributions are standardized by the transformation $U_n = F_n(X)$, where $F_n$ is the empirical cumulative distribution of the random variable~$X$. Figure~\ref{F:symIndex}(A)  displays the bivariate symmetry index $S_n$ developed in \cite{GNQ12},
\[
S_n:= \int_0^1 \int_0^1 \left[ \hat{C}_n(u,v) - \hat{C}_n(v,u) \right]^2 \textrm{d}\hat{C}_n(u,v),
\]
where $\hat{C}_n$ is the empirical copula as in (\ref{E:empcop}).

The greatest asymmetry is observed between variables LBM and Hg, 
but in all cases the proposed test in \cite{GNQ12} does not reject the symmetry null hypothesis 
(the p-value between LBM and Hg is $0.72$).

\begin{figure}[ht]
\centering
\subfloat[]{\includegraphics[scale=0.62]{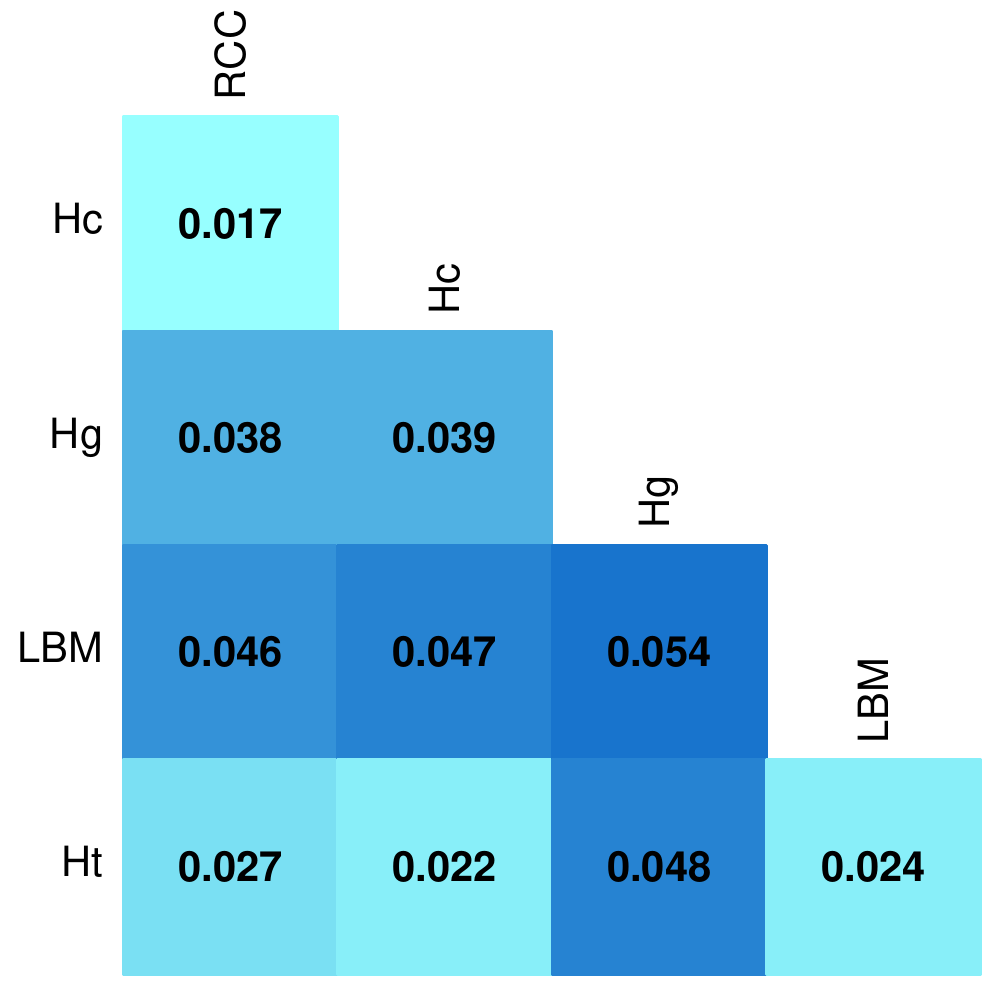}}
\subfloat[]{\includegraphics[scale=0.58]{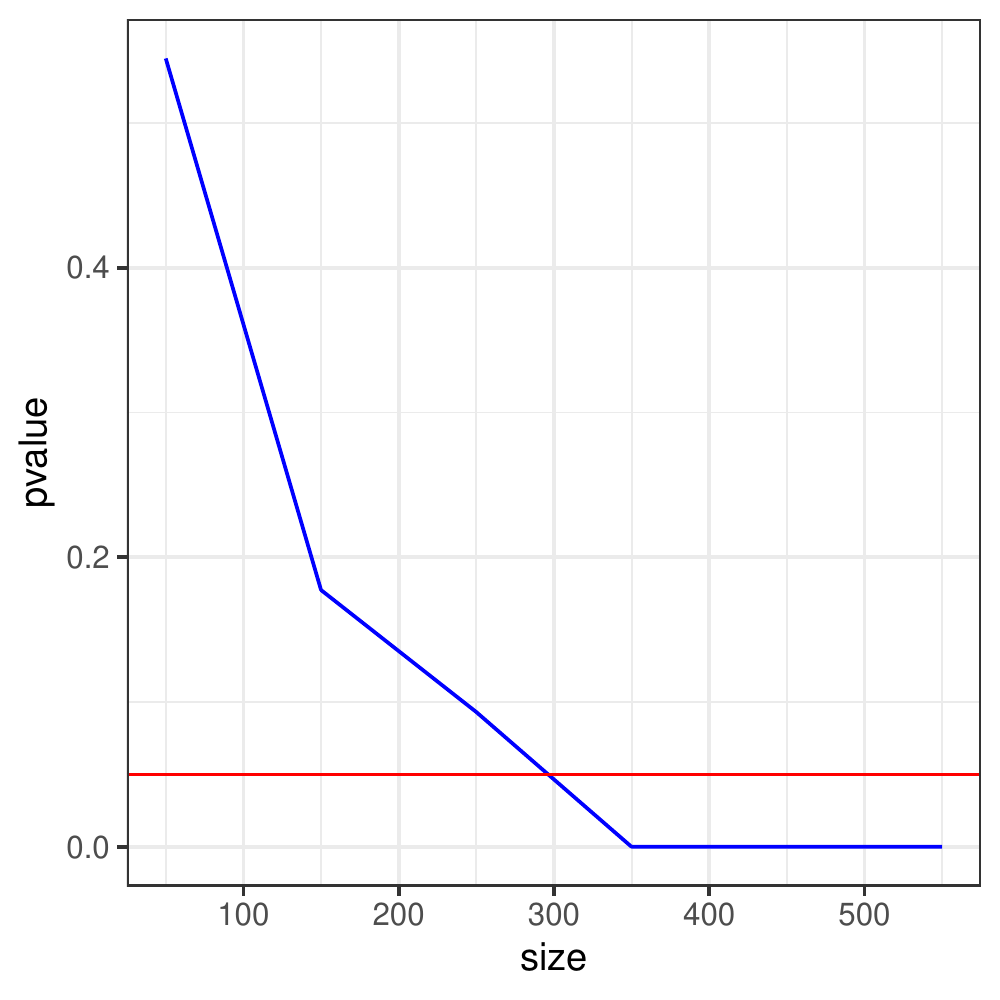}}
\caption{\small (A) Bivariate symmetry index values $S_n$ for all pairs of copulas among the 5 variables considered in the Biometric measurements. (B) P-values of the test (with statistic  $D_{n,\ell=50}$) for sub-sample of sizes  $n=50,150, \ldots 550$  in the Statlog Satellite dataset (horizontal red line $\alpha=0.05$).} 
\label{F:symIndex}
\end{figure}

We test the global five-dimensional symmetry of the copula. The components of each observation of the standardized sample are randomly permuted. From this permuted sample, we obtain the empirical distribution of the statistic $D_{n,\ell=50}$  under the exchangeability hypothesis (over 10,000 replicates). The p-value obtained for the sample is  $0.0126$. Therefore, as in \cite{BQ21}, the exchangeability hypothesis is rejected. 


\subsection{Statlog Satellite dataset}
 
The database consists of the multi-spectral values of pixels in $3\times3$ neighbourhoods in a satellite image, and is available in the UCI, Machine Learning Repository (\url{https://archive.ics.uci.edu/ml/machine-learning-databases/statlog/satimage/}) . The sample size is 6435 images. Each line contains the pixel values in the four spectral bands (converted to ASCII) of each of the 9 pixels in the $3\times3$ neighbourhood. That is, each observation is represented by $36$ attributes. In this example too the marginal distributions are clearly different, so, as in the previous example, the variables are standardized, and the statistic under the null hypothesis is determined to test exchangeability of the copula. The test  $D_{n,\ell=50}$ is performed for different sub-sample sizes $n=50,150, \ldots 550$  
(from the first to the n-th observation), similar to \cite{FGNV2012}. 
Figure~\ref{F:symIndex}(B)  displays the p-values  obtained for each~$n$. 
We observe that, with a sample size greater than $300$, the hypothesis of exchangeability is rejected.
 

\section{Conclusions and potential further developments}\label{S:conclusion}

We have shown how to exploit an extension of the Cram\'er--Wold theorem, Theorem~\ref{T:CWgen}, to use one-dimensional projections
to test for invariance of a probability measure under a group of linear transformations. As special cases, we obtain tests for
exchangeability and sign-invariant exchangeability.
The results in the simulations are really good, 
and the algorithms we propose are computationally extremely fast and adequate even for high-dimensional data. 
Regarding the comparison with the proposal in \cite{HS17}, 
Table \ref{Ta:pot} shows a relative efficiency of our method \emph{vis-\`a-vis} its competitor 
of order around 100\% for small values of the parameter~$\xi$. 
Moreover, the proposal in \cite{HS17} seems to be very hard to implement in high-dimensional spaces. 
We illustrate our method with a short study of two real-data examples.                                                                                     

As well as being effective, the methods developed in this article are quite flexible. 
Our results can be applied to the case of multivariate functional data, 
where we have vectors taking values in a Hilbert space or even a Banach space.
We can still employ the same algorithm as in \S\ref{SS:testing2}, 
using several random directions and finding the critical value by bootstrap,  
except that, in the infinite-dimensional case, we should choose the random directions 
(or the elements on the dual space) according to a non-degenerate gaussian measure. 

We have also shown that, even in the context of testing for $G$-invariance,
where one of the measures is a linear transformation of the other, 
Theorem~\ref{T:CWgen} remains sharp. This is the content of Theorem~\ref{T:sharpexch}.
However, if we have further \emph{a priori} knowledge of the distribution in question, then
Theorem~\ref{T:CWgen} can be improved.
For example, Heppes \cite[Theorem~$1'$]{He56} showed that,
if $P$ is discrete probability measures on $\RR^d$ supported on $k$ points,
and if  $Q$ is an arbitrary Borel probability measure on $\RR^d$ such that $\cE(P,Q)$ contains at least $(k+1)$ subspaces, no two of which are contained in any single hyperplane,
then $P=Q$. This opens the door to potential improvements of our procedures in the case where 
$P$ is known to be discrete, for example testing for the exchangeability of 
sequences of Bernoulli random variables.
This idea is explored further in \cite{FMR22}.

\section*{Acknowledgement}

Fraiman and Moreno  supported by grant FCE-1-2019-1-156054, Agencia Nacional de Investigaci\'on e Innovaci\'on, Uruguay. Ransford supported by grants from NSERC and the Canada Research Chairs program.

\bibliographystyle{plain}
\bibliography{biblist}

\end{document}